\documentclass[12pt,a4paper]{amsart}

\usepackage{amsmath}
\usepackage{amsfonts}
\usepackage{amssymb}
\usepackage{amsmath,amsbsy,amssymb,amscd}
\usepackage[cp1250]{inputenc}
\usepackage{graphicx} 
\usepackage{fancyhdr}
\newcommand{\n}{\mathbb{N}}
\newcommand{\z}{\mathbb{Z}}

\newcommand{\re}{\mathbb{R}}
\newcommand{\al}{\alpha}

\newcommand{\la}{\lambda}
\newcommand{\g}{\gamma}
\newcommand{\ep}{\varepsilon}
\newcommand{\de}{\delta}

\newtheorem{theorem}{Theorem}
\newtheorem{lemma}[theorem]{Lemma}     
\newtheorem{corollary}[theorem]{Corollary}
\newtheorem{proposition}[theorem]{Proposition}
\newtheorem{uw}[theorem]{Remark}
\title{BACKWARD RAUZY-VEECH ALGORITHM AND HORIZONTAL SADDLE CONNECTIONS\footnote{MSC classification: 37E05, 37E35}}
\author[P.~Berk]{Przemys\l aw Berk}
\address{Faculty of Mathematics and Computer Science, Nicolaus
	Copernicus University, ul. Chopina 12/18, 87-100 Toru\'n, Poland}
\address{Institut für Mathematik, Universität Zürich, Winterthurerstrasse 190, CH-8057 Zürich, Switzerland}
\email{zimowy@mat.umk.pl}
\begin{document}
	\baselineskip=14pt \maketitle
\begin{abstract} We study the combinatorial and dynamical properties of translations surfaces with horizontal saddle connections from the point of view of backward Rauzy-Veech induction. Namely, we prove that although the horizontal saddle connections do not rule out existence of the infinite orbit under backward Rauzy-Veech algorithm, they disallow the $\infty$- completeness of such orbit. Furthermore, we prove that if an orbit under backward Rauzy-Veech algorithm is infinite, then the minimality of the horizontal translation flow is equivalent to the eventual appearance of all horizontal saddle connections as sides of the polygonal represenation of a surface.
\end{abstract}

\vspace{1cm}
The main goal of this note is to study the relations between horizontal saddle connections and the combinatorics of the inverse Rauzy-Veech algorithm for translation surfaces as well as dynamics of the horizontal translation flows. In \cite{MUY} (Proposition 4.3) Marmi, Ulcigrai and Yoccoz prove that if a translation surface does not have horizontal saddle connections, then its backward Rauzy-Veech induction  orbit is indefinitely well-defined and $\infty$-complete, that is every symbol is a backward winner infinitely many times. In the same article the authors pose a question, whether these two conditions are equivalent. We answer affirmatively to this question in Theorem \ref{noinfty}. The proof utilizes only combinatorics and geometry of translation surfaces.

However, before proving Theorem \ref{noinfty}, we prove Proposition \ref{exist} which states that, typically, possessing horizontal saddle connections  does not rule out that the backward orbit with respect to the inverse Rauzy-Veech algorithm is well defined. Moreover, in Theorem \ref{min} we prove that appearance of horizontal connections as sides of polygonal representations of translations surfaces is closely tied to the minimality of the horizontal  translation flow. More precisely, we show that the horizontal translation flow is minimal if and only if all (if any) horizontal saddle connections appear as sides of a polygonal representation of a surface after applying a finite number of backward Rauzy-Veech induction steps.

\textbf{Acknowledgments:} The author would like to thank Corinna Ulcigrai for pointing out the problem and her continuous support and Frank Trujillo for many useful remarks. The research that lead to this result was supported by\\ \emph{Swiss National Science Foundation} Grant $200021\_188617/1$ and\\ \emph{Narodowe Centrum Nauki }Grant OPUS $14 2017/27/B/ST1/00078$.

\section{Interval exchange transformations and translations surfaces}
We recall first basic notions and properties related to IETs and translation surfaces. Let $\mathcal{A}$ be an alphabet of $\#\mathcal A\ge 2$ elements. 
For more information and basic properties, including the ergodic properties of interval exchange transformations, translation surfaces and Rauzy-Veech algorithm we refer the reader e.g. to \cite{Viana} and \cite{Yoccoz}. 

Let 
\[
\begin{split}
S_0^{\mathcal A}:=&\{\pi=(\pi_0,\pi_1):\mathcal A\to\{1,\ldots, \#\mathcal A\}\times \{1,\ldots, \#\mathcal A\};\\ &\pi_1\circ\pi_0^{-1}\{1,\ldots, k\}=\{1,\ldots, k\}\Rightarrow\ k=\#\mathcal A\}
\end{split}
\]
be the set of irreducible permutations, where $\pi_0$ and $\pi_1$ are bijections. Let us also denote by $\re_{>0}^{\mathcal A}$ the set of all $d$-dimensional positive real vectors and for every $\la\in \re_{>0}^{\mathcal A}$ let $|\la|:=\sum_{\al\in\mathcal A}\la_\al$.

An \emph{interval exchange transformation on $[0,|\la|)$ (IET)} $T=(\pi,\la)\in S_0^{\mathcal A}\times  \re_{>0}^{\mathcal A}$ is a bijective piecewise translation, where the intervals 
\[I_\al:=\left[\sum_{\beta\in\mathcal A; \pi_0(\beta)<\pi_0(\al) }\la_\beta, \sum_{\beta\in\mathcal A; \pi_0(\beta)\le \pi_0(\al) }\la_\beta\right)\text{ for }\al\in\mathcal A
	\]
	are rearranged inside $[0,|\la|)$ with respect to the permutation $\pi$. More precisely, for every $\al\in\mathcal A$, we have
\[
T(x)=x+\de_\al\quad\text{if}\quad x\in I_\al,
\]
where
\[
\de_\al=\sum_{\beta\in\mathcal A;\ \pi_1(\beta)<\pi_1(\al)}\la_\al-\sum_{\beta\in\mathcal A;\ \pi_0(\beta)<\pi_0(\al)}\la_\al.
\]
Note that $T$ preserves Lebesgue measure.

We denote by $\Omega^\pi=[\omega_{\al\beta}]_{\al,\beta}$ the associated  \emph{translation matrix}, with coefficients given by
\[
\omega_{\al\beta}:=\begin{cases}
+1 & \text{ if }\pi_0(\al)<\pi_0(\beta)\text{ and }\pi_1(\al)>\pi_1(\beta);\\
-1 & \text{ if }\pi_0(\al)>\pi_0(\beta)\text{ and }\pi_1(\al)<\pi_1(\beta);\\
\ \ 0 & \text{ otherwise.}
\end{cases}
\]
Then, if $\de:=[\de_\al]_{\al\in\mathcal A}$, we get
\[
\de=\Omega_{\pi}\cdot\la.
\]

On the space $S_0^{\mathcal A}\times  \re_{>0}^{\mathcal A}$ we consider an operator $R$ called \emph{Rauzy-Veech induction}, defined as $R(\pi,\la)=(\pi^1,\la^1)$, where $(\pi^1,\la^1)$ is the first return map of $(\pi,\la)$ to the interval $[0,|\la|-\min\{\la_{\pi_0^{-1}(d)},\la_{\pi_1^{-1}(d)}\})$. 
If  $\la_{\pi_0^{-1}(d)}>\la_{\pi_1^{-1}(d)}$ we say that $R$ is of "top" type and we say that it is of "bottom" type if $\la_{\pi_0^{-1}(d)}<\la_{\pi_1^{-1}(d)}$. We denote the symbol corresponding to the longer interval as $w$ (the winner) and to the shorter one as $l$ (the loser).

The map $R(\pi,\la)$ is properly defined as an interval exchange transformation of $d$ intervals if and only if $\la_{\pi_0^{-1}(d)}\neq\la_{\pi_1^{-1}(d)}$. Keane \cite{Keane} gave an equivalent condition on $(\pi,\la)$, for the iterations of Rauzy-Veech induction to be defined indefinitely. More precisely, we say that IET $T$ satisfies \emph{Keane's condition} if for every two discontinuities $a$ and $b$ of $T$ equality $T^n(a)=b$ for some $n\in\n$ implies $n=1$, $a=T^{-1}(0)$ and $b=0$.  In particular, if the vector $\la$ is rationally independent, that is for every choice of $c_\al\in\z,\ \al\in\mathcal A$ we have
\[
\sum_{\al\in\mathcal A}c_\al\la_\al=0\ \Rightarrow\ c_\al=0\text{ for every }\al\in\mathcal A,
\]
then $(\pi,\la)$ satisfies Keane's condition. When it is well defined, we denote $R^n(\pi,\la)=(\pi^n,\la^n)$ for every $n\in\n$. We say that the orbit of $(\pi,\la)$ via Rauzy-Veech induction is \emph{$\infty$-complete} if every symbol in $\mathcal A$ appears infinitely many times in the sequence of winners $\{w^n\}$.

Note that $\la^1=A^1(\pi,\la)\la$, where a matrix $A^1(\pi,\la)$ is defined in the following way
\[
A^1_{\al\beta}=\begin{cases}
1 & \text{ if }\al=\beta;\\
-1 & \text{ if }\al=w \text{ and } \beta=l;\\
0 & \text{ otherwise.}
\end{cases}
\]
Inductively, for every $n\in\n$ we define 
\[
A^n(\pi,\la)=A^1(\pi^{n-1},\la^{n-1})A^{n-1}(\pi,\la). 
\]
Then $\la^n=A^n(\pi,\la)\la$. We will refer  to $A^n(\pi,\la)$ as \emph{Rauzy-Veech matrices}. Note that for every $n\in\n$, the matrix $(A^n(\pi,\la))^{-1}$ is non-negative.

For every $\pi \in S_0^{\mathcal A}$ let 
\[
\begin{split}
\Theta^{\mathcal A}=\Theta^{\mathcal A}({\pi})&=\Big\{\tau\in\re^{\mathcal A};\ \sum_{\al\in\mathcal A;\pi_0(\al)\le k}\tau_{\al}>0\ \text{ and }\\&   \sum_{\al\in\mathcal A;\pi_1(\al)\le k}\tau_{\al}<0\ \text{ for every }\ k\in\{1,\ldots,d-1\}\Big\}.
\end{split}
\]
Then  every $(\pi,\la,\tau)\in S_0^{\mathcal A}\times \Lambda^{\mathcal A}\times \Theta^{\mathcal A}$\footnote{Note that this space is not really a product space since $\Theta^{\mathcal A}$ depends on $\pi$ and thus $S_0^{\mathcal A}\times \Lambda^{\mathcal A}\times \Theta^{\mathcal A}=\bigcup_{\pi\in S_0^{\mathcal A}} \{\pi\}\times \Lambda^{\mathcal A}\times \Theta^{\mathcal A}(\pi)$. However, we shall use this notation for simplicity.} may be see as a \emph{translation surface} as follows. More precisely, first we consider two broken line segments in $\mathbb C$
\[\bigcup_{k=1}^d\left[\sum_{\al\in\mathcal A;\  \pi_0(\al)<k}(\la_\al+i\tau_\al),\sum_{\al\in\mathcal A;\ \pi_0(\al)\le k}(\la_\al+i\tau_\al)\right]\ %\text{for $k=1,\ldots,d$},
\]  
and 
\[\bigcup_{i=1}^d\left[\sum_{\al\in\mathcal A;\ \pi_1(\al)<k}(\la_\al+i\tau_\al),\sum_{\al\in\mathcal A;\ \pi_1(\al)\le k}(\la_\al+i\tau_\al)\right].%\text{for $k=1,\ldots,d$}.
\]
Then we identify the segments corresponding to the same symbols via parallel translation. 
The endpoints of these segments are the \emph{singularity points} of the surface and are denoted by $\Sigma$ (which may be conical singularities as well as marked points). Note that some of the points $S\in\Sigma$ may correspond to many vertices of the polygon given by $(\pi,\la,\tau)$ before identification. For $0\le k\le\#\mathcal A$ we denote
\[\begin{split}
&a(\pi,\la,\tau,k):=\sum_{\al\in\mathcal A;\ \pi_0(\al)\le k}(\la_\al+i\tau_\al) \text{ and }\\
&b(\pi,\la,\tau,k):=\sum_{\al\in\mathcal A;\ \pi_1(\al)\le k}(\la_\al+i\tau_\al),
\end{split}
\]
the vertices of the polygon given by $(\pi,\la,\tau)$ (note that $a(\pi,\la,\tau,\#\mathcal A)=b(\pi,\la,\tau,\#\mathcal A)$ and $a(\pi,\la,\tau,0)=b(\pi,\la,\tau,0)=0$). From now on, for every $\al\in\mathcal A$, we will refer to the segment with endpoints $a(\pi,\la,\tau,\pi_0(\al)-1)$ and  $a(\pi,\la,\tau,\pi_0(\al))$ as well as to the segments $b(\pi,\la,\tau,\pi_0(\al)-1)$ and  $b(\pi,\la,\tau,\pi_0(\al))$  as segments \emph{corresponding} to $\al$.

			\begin{figure}
	\includegraphics[scale=0.5]{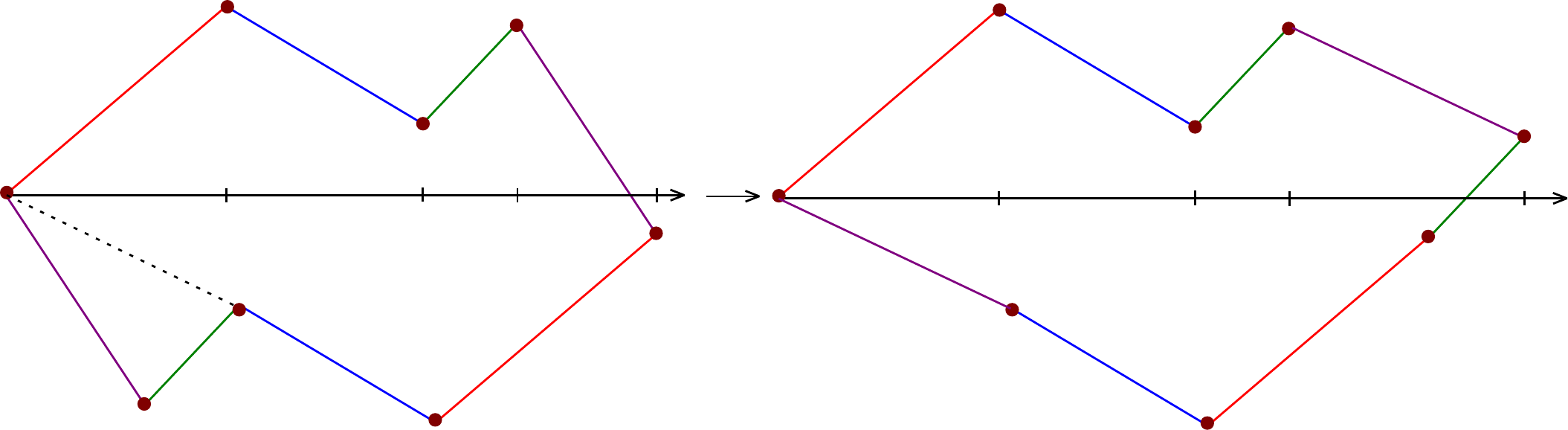}
	%		$\quad\quad\quad\quad$
	%		\includegraphics[scale=0.25]{drawing-1.pdf}
	\caption{A translation surface and one step of backward Rauzy-Veech induction. The parallel segments are identified via translation. The winning segment is the one which is first crossed by a rightward separatrix starting from $(0,0)$.}\label{SF}
\end{figure}

On a surface  $(\pi,\la,\tau)$ we consider a \emph{translation flow}, that is the flow which moves every non-singular point with unit speed in a fixed direction. In this note we mostly use the horizontal rightward direction and refer to such flows simply as \emph{``horizontal flows"} $\{T_t\}_{t\in\re}$. If the orbit of some point hits a singularity, then we call such an orbit a \emph{separatrix} of $\{T_t\}_{t\in\re}$. If the singularity is hit in negative time, then we say that the separatrix is \emph{rightward} and if it is hit in positive time, then we say that the separatrix is \emph{leftward}. A \emph{saddle connection} is a separatrix which is both rightward and leftward.  

The following fact concerning minimality, or rather its corollary (Cor. \ref{minorb}), would be of later use.
\begin{proposition}[see Theorem 3.13 in \cite{Viana}]
Every translation surface $(\pi,\la,\tau)$ admits a decomposition into finitely many maximal subsets $D_j$, $j=1,\ldots,k$, invariant under the action of horizontal flow, such that the restriction of the horizontal flow to $D_j$ for every $j=1,\ldots,k$ is either periodic or minimal.
\end{proposition}
\begin{corollary}\label{minorb}
	The horizontal flow on $(\pi,\la,\tau)$ is minimal if and only if there exists one half-orbit which is dense.
\end{corollary}

We extend the definition of Rauzy-Veech induction into the space $S_0^{\mathcal A}\times \Lambda^{\mathcal A}\times \Theta^{\mathcal A}$. Namely $\mathbf{R}(\pi,\la,\tau):=(\pi^1,\la^1,\tau^1)$, where
\[
(\pi^1,\la^1)={R}(\pi,\la)\ \text{ and }\ \tau^1=A^1(\pi,\la)\tau.
\]
Thus, $\mathbf{R}(\pi,\la,\tau)$ is well defined if and only if $R(\pi,\la)$ is well defined. 
%In particular, if $\la^n_{\pi_0^{-1}(d)}=\la^n_{\pi_1^{-1}(d)}$ then the Rauzy-Veech induction stops.
We define the \emph{type} of  $(\pi,\la,\tau)$ as that of $(\pi,\la)$.

Note that $R$ is not a invertible map. Indeed, every $(\pi,\la)\in S_0^{\mathcal A}\times\Lambda^{\mathcal A}$ has exactly two preimages. The map $\mathbf{R}$ on the other hand is invertible whenever $(\pi,\la,\tau)$ satisfies $\sum_{\al\in\mathcal A}\tau_{\al}\neq0$. We can thus consider the \emph{backward Rauzy-Veech induction} $\mathbf{R^{-1}}$. We say that $(\pi,\la,\tau)$  is of the \emph{backward ``top" type} if  $\sum_{\al\in\mathcal A}\tau_{\al}<0$ and is of the \emph{backward ``bottom" type} if $\sum_{\al\in\mathcal A}\tau_{\al}>0$. 

Moreover, if  $(\pi,\la,\tau)$ is of backward "top" type, we say that $\pi_0^{-1}(d)$ is a \emph{backward winner} and $\pi_1^{-1}(d)$ is a \emph{backward loser}. Analogously if  $(\pi,\la,\tau)$ is of backward "bottom" type, we say that $\pi_1^{-1}(d)$ is a \emph{backward winner} and $\pi_0^{-1}(d)$ is a \emph{backward loser}. Although the following result is well known, we present its short proof for the sake of completeness.
\begin{lemma}
	We have that
	\begin{equation}\label{type}
	\begin{split}
	(\pi,\la,\tau)&\ \text{is of the ``backward top" (``backward bottom") type}\\ &\Leftrightarrow  \mathbf{R^{-1}}(\pi,\la,\tau)\ \text{is of ``top" (``bottom") type.}
	\end{split}
	\end{equation}
	Moreover
	\begin{equation}
	\begin{split}
	\al&\text{ is a backward winner of }\mathbf{R}^{-1}\text{ for }(\pi,\la,\tau)\\
	&\Leftrightarrow \al\text{ is a winner of  }\mathbf{R}\text{ for }\mathbf{R}^{-1}(\pi,\la,\tau).
	\end{split}
	\end{equation}
\end{lemma}
\begin{proof}
	Assume that  $(\pi,\la,\tau)$ is of backward top type and $\al$ is a backward winner that is $\al=\pi_0^{-1}(d)$ (the backward bottom type case is done analogously). Denote also $\beta= (\pi_1)^{-1}(\pi_1(\al)+1)$. Then $\pi_0^{-1}(\al)=\pi_0(\al)=\#\mathcal A$ and $\pi_1(\beta)=\mathcal A$. Moreover
	\[
	\la^{-1}_\al=\la_\al+\la_\beta>\la_\beta=\la^{-1}_\beta.
	\]
	Thus $(\pi,\la,\tau)$ is of (forward) top type and $\al$ is the winner.
	\end{proof}
\begin{uw}\label{hitby0}
Note that $\al\in\mathcal A$ is a backward winner iff the segment corresponding to $\al$ is the first  segment hit by the rightward separatrix starting from the point $(0,0)$ in the polygonal representation of $(\pi,\la,\tau)$ (see Figure \ref{SF}).
	\end{uw}

It is  easy to see that $\mathbf{R}^{-n}(\pi,\la,\tau)$ is properly defined for every $n\in\n$ if $\tau$ is a rationally independent vector. In particular this together with Lemma \ref{orb0} implies that in a surface $(\pi,\la,\tau)$ rightward separatrix starting from $(0,0)$ is not a saddle connection. In Lemma \ref{exist} we show that in order to define an infinite orbit of a backward Rauzy-Veech induction, the condition on rational independence can be significantly weakened.

We define a Rauzy-Veech matrix associated to $\mathbf{R}^{-1}$ at point $(\pi,\la,\tau)$ by
\[
A^{-1}(\pi,\la,\tau):=\left(A^1\left(\mathbf{R}^{-1}(\pi,\la,\tau) \right)\right)^{-1},
\]
and analogously as in the forward case for every $n\in\n$ we define
\[
A^{-n}(\pi,\la,\tau):=\left(A^n\left(\mathbf{R}^{-n}(\pi,\la,\tau) \right)\right)^{-1},
\]
whenever $\mathbf{R}^{-n}$ is properly defined. If $\mathbf{R}^{-n}(\pi,\la,\tau)=(\pi^{-n},\la^{-n},\tau^{-n})$ then we have
\begin{equation}\label{matback}
\la^{-n}=A^{-n}(\pi,\la,\tau)\la \quad\text{and}\quad \tau^{-n}=A^{-n}(\pi,\la,\tau)\tau.
\end{equation}

If $\sum_{\al\in\mathcal A}\tau^{-n}_\al=0$ for some $n\in\n$ then the backward Rauzy-Veech induction stops, i.e. $\mathbf R^{-n-1}(\pi,\la,\tau)$ is not well defined. If on the other hand $\mathbf R^{-n}(\pi,\la,\tau)$ is well defined for every $n\in\n$ and each symbol is a backward winner infinitely many times then we say that \emph{$(\pi,\la,\tau)$ has $\infty$-complete backward Rauzy-Veech induction orbit}.

Note that since $(A^n(\pi^{-n},\la^{-n}))^{-1}$ is non-negative, we have
\begin{equation}
A^{-n}(\pi,\la,\tau)\quad\text{is non-negative for every}\ n\in\n.
\end{equation}
In particular 
\begin{equation}\label{nondec}
\min_{\al,\beta\in\mathcal A}A_{\al\beta}^{-n}(\pi,\la,\tau)\quad\text{is non-decreasing as}\ n\to\infty.
\end{equation}
In particular 
\begin{equation}\label{lalength}
\lim_{n\to\infty}|\la^{-n}|=\infty.
\end{equation}

The following result is stated as a Remark 4.2 in \cite{MUY}, however, due to its importance in this article, we present its short proof.
\begin{lemma}\label{orb0}
	The backward Rauzy Veech induction is defined indefinitely on $(\pi,\la,\tau)$ if and only if the horizontal rightwards separatrix starting from point $(0,0)$ in $(\pi,\la,\tau)$ is infinite, i.e. it is not a horizontal saddle connection. 
\end{lemma}
\begin{proof}
	If $(\pi^{-n'},\la^{-n'},\tau^{-n'})$ is not properly defined for some $n'\in\n$ then $\sum_{\al\in\mathcal A}\tau_\al^{-n'+1}=0$ and the horizontal interval with endpoints at $(0,0)$ and $(\sum_{\al\in\mathcal A}\la_\al^{-n'+1},0)$ is a horizontal saddle connection. 
	
	On the other hand if there is a horizontal saddle connection starting at $(0,0)$ of length $\ell>0$ and $(\pi^{-n},\la^{-n},\tau^{-n})$ is well defined for every $n\in\n$, then by \eqref{lalength} there exists $n'\in\n$ such that $\sum_{\al\in\mathcal A}\la_\al^{n'}>\ell$. This implies however that the whole saddle connection is in the interior of $(\pi^{-n'},\la^{-n'},\tau^{-n'})$ seen as a polygon, in particular this applies to the right-hand side endpoint of the saddle connection. This is however a contradiction since the singularities of $(\pi,\la,\tau)$ can be only the vertices and they do not belong to the horizontal line.
\end{proof}

The surface $(\pi,\la,\tau)\in S_0^{\mathcal A}\times \Lambda^{\mathcal{A}}\times \Theta^{\mathcal A}$ is  alternatively considered via \emph{zippered rectangles} representation $(\pi,\la,h)\in S_0^{\mathcal A}\times \Lambda^{\mathcal A}\times \re_{>0}^{\mathcal A}$, that is one considers a Poincar\'e return map of the vertical translation flow to the rightward separatrix segment of length $|\la|$, originating from the point $(0,0)$. Then the first return map is an interval exchange transformation $(\pi,\la)$ and the first return times are constant on each exchanged interval and given by the \emph{height vector}  $h=\Omega_{\pi}^T\tau$. Then a \emph{ rectangle} associated to the symbol $\al\in\mathcal A$ is the set \[\left[\sum_{\al\in\mathcal A;\ \pi_0(\beta)<\pi_0(\al)}\la_\beta,\sum_{\al\in\mathcal A;\ \pi_0(\beta)\le\pi_0(\al)}\la_\beta\right)\times [0,h_\al).\] 

Moreover, the sides of rectangles are divided into parts and identified with each other with a proper rearrangement. The segments which are identified are referred to as ``zips". The rectangles together with the zips form a \emph{zippered rectangles} representation of a surface (see Figure \ref{zip}).

One can prove by a simple induction that the points  in $\Sigma$  always belong to the left-hand edge of the  rectangles, that is in the sets of the form $\{\sum_{\al\in\mathcal A;\ \pi_0(\beta)<\al}\la_\beta\}\times [0,h_\al)$ for every $\al\in\mathcal A$.

			\begin{figure}
	\includegraphics[scale=0.5]{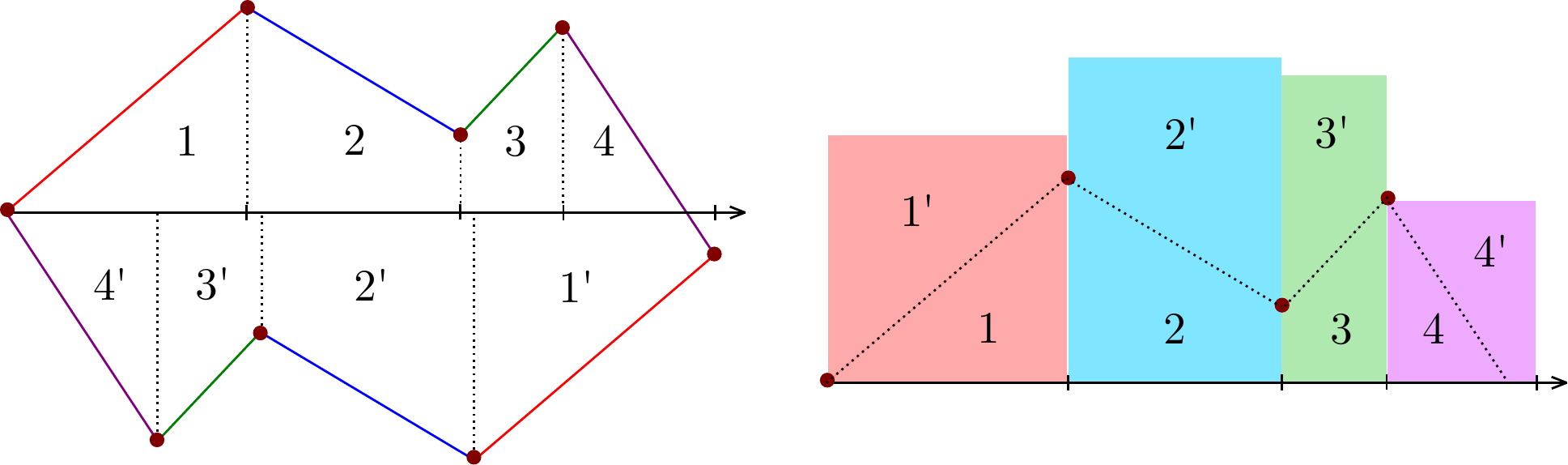}
	%		$\quad\quad\quad\quad$
	%		\includegraphics[scale=0.25]{drawing-1.pdf}
	\caption{A polygonal and a zippered rectangle representation of a surface.}\label{zip}
\end{figure}

It is possible to obtain zippered rectangles by dividing the polygonal construction into smaller polygons and rearranging them. Hence each rectangle properly locally parametrizes the surface $(\pi,\la,\tau)$. For every $n\in\z$ we denote $h^n=\Omega_{\pi^n}\tau^n$, that is  rectangle height vector corresponding to the surface $(\pi^n,\la^n,\tau^n) $ obtained by $n$ steps of Rauzy-Veech induction.

%Let also
%\[
%\Lambda^{\mathcal A}:=\{\la\in\re_{>0}^{\mathcal A};\ \sum_{\al\in\mathcal A}\la_\al=1 \}
%\]
%be a $d-1$-dimensional unit simplex. 
\section{The existence of the backward orbit}
Note that the backward Rauzy-Veech induction algorithm stops  the orbit of $(\pi,\la,\tau)\in \{\pi\}\times \Lambda^{\mathcal A}\times \Theta^{\mathcal A}$ when $\sum_{\al\in\mathcal A}\tau_\al=0$. This implies that there is a horizontal connection between 0 and $\sum_{\al\in\mathcal A}\la_\al+i\tau_\al$ and that the latter  is also an element of $\Sigma$ (see Lemma \ref{orb0}). The following result states that the existence of horizontal saddle connections does not imply that the Rauzy-Veech induction orbit stops.

\begin{proposition}\label{exist1}
	There exist translation surfaces with horizontal saddle connections, whose orbit under the action of backward Rauzy-Veech induction is defined indefinitely.
\end{proposition}
To prove the above result we will now see that, in general, it is enough to pick $\tau_\beta$ for some $\beta\in\mathcal A$ appropriately for the iterations of backward Rauzy-Veech induction on $(\pi,\la,\tau)$ to be defined infinitely many times. 
\begin{lemma}\label{exist}
	Let $\beta\in\mathcal A$. Assume that $\{\tau_\al\}_{\al\in\mathcal A}$ is such that for every choice of integer numbers $c_{\al}\in \mathbb{Z},\ \al\in\mathcal A$ we have
	\[
	\sum_{\al\in\mathcal A}c_\al\tau_\al=0\ \Rightarrow\ c_\beta=0.
	\]
	Then the backward Rauzy-Veech induction iterates are defined indefinitely. 
\end{lemma}
\begin{proof}
	Note first that by assumptions taking $c_\al=1$, $\al\in\mathcal A$, we have
	\[
\sum_{\al\in\mathcal A}\tau_\al\neq 0.
	\]
	Thus $(\pi^{-1},\la^{-1},\tau^{-1})$ is properly defined. 
	
	We proceed by induction. Assume that $(\pi^{-n},\la^{-n},\tau^{-n})$ for $n\ge 1$ is properly defined and let $B=A^{-n}(\pi,\la,\tau)$ be a Rauzy-Veech matrix of the $n$ backward steps of induction. Then in particular $B$ is a non-negative matrix and $B_{\beta\beta}\ge 1$. Indeed, the coefficients of the Rauzy-Veech matrix are non-decreasing (see \eqref{nondec}) and $A^{-1}(\pi,\la,\tau)$ has ones on the diagonal. 
	
	We claim that $(\pi^{-n-1},\la^{-n-1},\tau^{-n-1})$ is properly defined. Assume otherwise, that is 
	\begin{equation}\label{tau0}
	\sum_{\al\in\mathcal A}\tau_{\al}^{-n}=0.
	\end{equation}
	Since $\tau^{-n}=B\tau$, we have
	\[
	\sum_{\al\in\mathcal A} c_\al\tau_\al=0,
	\]
	where $c_\al=\sum_{\g\in\mathcal A}B_{\g\al}$. In particular, since $B_{\beta\beta}\ge 1$ and $B$ is a non-negative matrix, we have $c_{\beta}>0$. This, together with \eqref{tau0}, yields a contradiction with the assumption of the lemma.
	
		\end{proof}
	\begin{figure}
		\includegraphics[scale=0.5]{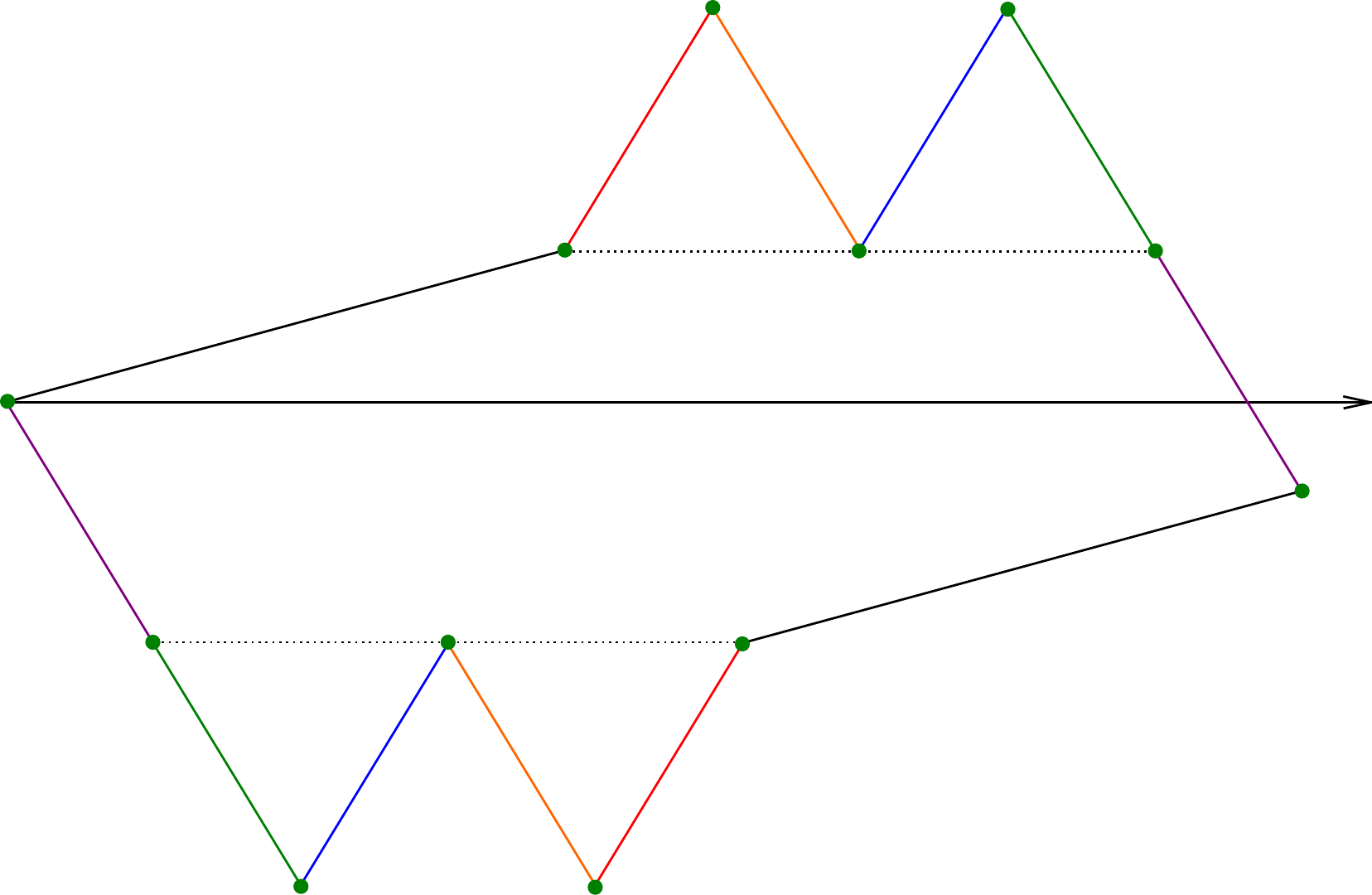}
		%		$\quad\quad\quad\quad$
		%		\includegraphics[scale=0.25]{drawing-1.pdf}
		\caption{An example of a surface with horizontal saddle connections (dotted lines) for which backward Rauzy-Veech induction orbit is well defined - all segments besides the black segment are identical up to reflection and the vertical coordinates of the black segment and the remaining segments are rationally independent.}\label{saddledef}
	\end{figure}
\begin{proof}[Proof of Proposition \ref{exist1}]
	Consider $(\pi,\la,\tau)\in\{\pi\}\times \Lambda^{\mathcal A}\times\Theta^{\mathcal A}$ where 
	\begin{enumerate}
		\item $\tau_{\pi_0^{-1}(1)}\in\re\setminus\mathbb{Q}$;
		\item $\tau_{\g}\in\mathbb{Q}$ for $\g\neq \pi_0^{-1}(1)$;
		\item $\tau_{\pi_0^{-1}(2)}=-\tau_{\pi_0^{-1}(3)}>0$.
	\end{enumerate}
	Then $(\pi,\la,\tau)$ satisfies assumptions of Proposition \ref{exist} with $\beta=\pi_0^{-1}(1)$, but the segment \[[\la_{\pi_0^{-1}(1)}+i\tau_{\pi_0^{-1}(1)},\sum_{i=1,2,3}\la_{\pi_0^{-1}(i)}+i\tau_{\pi_0^{-1}(i)}]\]
	is a horizontal saddle connection (see Figure \ref{saddledef}).
	\end{proof}

\section{Horizontal connections prevent $\infty$-completeness}
We saw in Corollary \ref{exist1} that a horizontal saddle connection does not necessarily prevent a proper definition of an infinite backward Rauzy-Veech orbit. We shall see that it does prevent $\infty$-completeness. 
As the reader will see, it  follows from the proof of Theorem \ref{noinfty} that horizontal saddle connections ``freeze" some coordinates in the sense that they stop winning after finite number of steps of backward Rauzy-Veech algorithm. Before proving Theorem \ref{noinfty}, we present an easy condition to prevent a symbol from winning.
\begin{lemma}\label{win0}
	Let $(\pi,\la,\tau)\in S_0^{\mathcal A}\times\Lambda^{\mathcal A}\times \Theta^{\mathcal A}$. If $\tau_\beta=0$ for some $\beta\in\mathcal A$ then $\beta$ cannot be a backward winner. 
\end{lemma}
\begin{figure}[h]
	\includegraphics[scale=0.5]{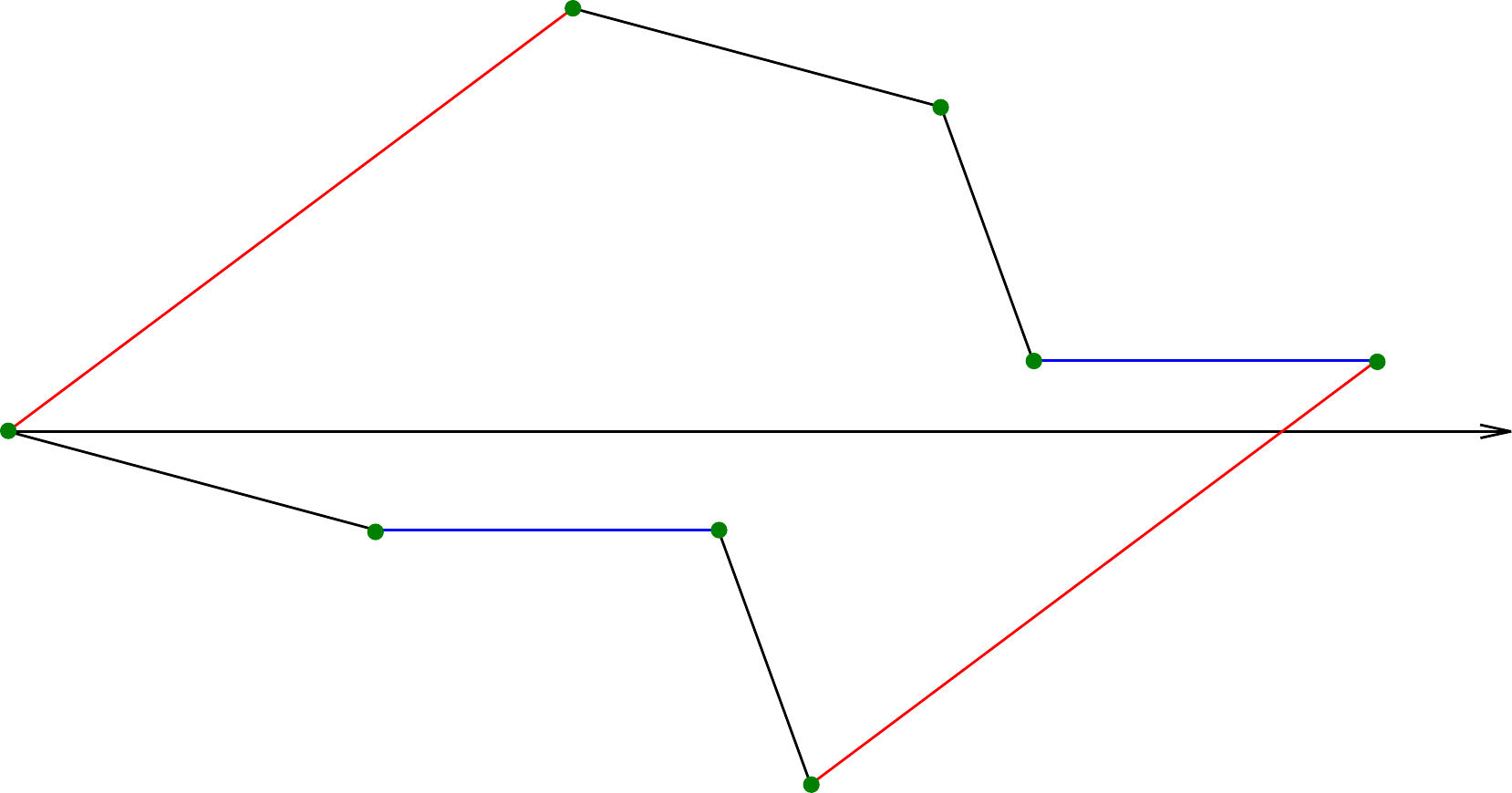}
	%		$\quad\quad\quad\quad$
	%		\includegraphics[scale=0.25]{drawing-1.pdf}
	\caption{A horizontal segment cannot cross transversally $x$-axis, thus the corresponding symbol cannot be a backward winner.}
\end{figure}
\begin{proof}
	By accelerating backward Rauzy-Veech algorithm if necessary we can assume that $\beta$ is a winner at step $1$ of  backward Rauzy-Veech induction and that $(\pi,\la,\tau)$ is of backward top type, that is $\beta=\pi_0^{-1}(\#\mathcal A)$ (the case when it is of backward bottom type is symmetric). Then $\sum_{\al\in\mathcal A}\tau_\al=0$. However, since $\tau_\beta=0$, we get $\sum_{\al\in\mathcal A\setminus\{\pi_0^{-1}(d)\}}\tau_\al=0$. This is a contradiction with the definition of $\Theta^{\mathcal A}$.
\end{proof}
%We are going to use the following crucial observation.
%\begin{theorem}[see (MUY)]\label{muy}
%	Let $(\pi,\la,\tau)\in \{\pi\}\times \Lambda^{\mathcal A}\times \Theta^{\mathcal A}$ be such that $(\pi^{-n},\la^{-n}, \tau^{-n}) $ is properly defined. Then the path of the backward Rauzy-Veech induction of $(\pi,\la,\tau)$ is infinitely many times of backward both top and backward bottom type.
%	\end{theorem}
We have the following property of $\infty$-complete orbits.
\begin{theorem}[see subsection 1.2.4 in \cite{MMY}]
	If the path of the forward Rauzy-Veech induction of $(\pi,\la,\tau)$ is properly defined and $\infty$-complete then the Rauzy-Veech matrix obtained after each but one  symbol has won at least $2d-3$ times is positive.
\end{theorem}
\noindent	By \eqref{type} we have the following.
	\begin{corollary}\label{signchange}
		If the path of the backward Rauzy-Veech induction of $(\pi,\la,\tau)$ is properly defined and $\infty$-complete then the Rauzy-Veech matrix obtained after each symbol but one was a backward winner at least $2d-2$ times is positive.
	\end{corollary}
\begin{theorem}\label{noinfty}
	Assume that $(\pi,\la,\tau)\in \{\pi\}\times \Lambda^{\mathcal A}\times \Theta^{\mathcal A}$ has a horizontal saddle connection. Then the backward Rauzy-Veech induction orbit of $(\pi,\la,\tau)$ is not  backward $\infty$-complete.
	\end{theorem}
\begin{proof}
If the backward Rauzy-Veech induction stops then in particular the path is not $\infty$-complete. Assume then that the backward Rauzy-Veech induction orbit is defined indefinitely.

	We will proceed by contradiction that is we assume that the backward Rauzy-Veech induction orbit is backward $\infty$-complete. Moreover, assume that $(\pi,\la,\tau)$ has a horizontal saddle connection $\ell$ of length $|\ell|$. Let $S\in\Sigma$ be its left endpoint. Then in the polygonal representation $\ell$ can be seen as a segment starting at one of the vertices $X$ of the polygon corresponding to $(\pi,\la,\tau)$, going rightwards. 
	
	In view of Corollary \ref{signchange}, there exists $n_0\in\n$ such that for $n>n_0$ the Rauzy-Veech matrix $A^{-n}$ satisfies
	\begin{equation}\label{>l}
	\min_{\al,\beta\in\mathcal A}A^{-n}_{\al,\beta}>|\ell|(\min_{\al\in\mathcal A} |\la_{\al}|)^{-1}.
	\end{equation}
	Indeed, backward $\infty$-completeness implies that for $n_0$ large enough each symbol won at least $2d-2$ times in the backward Rauzy-Veech induction path of length $n_0$. Corollary \ref{signchange} implies then that for  $n_0$ the Rauzy-Veech matrix $A^{-n_0}$ is positive. By repeating this process for $(\pi^{-n_0},\la^{-n_0},\tau^{n_0})$ and then proceeding inductively and using the fact that the entries of the product of $N$ positive integer matrices are not smaller than $(\#\mathcal A)^{N-1}$ we obtain $n$ in  \eqref{>l}. In particular, since $\la^{-n}=A^{-n}\la$, we have obtained that
	\begin{equation}\label{biglambda}
	|\la^{-n}_\al|>|\ell|\quad\text{for every}\ \al\in\mathcal A.
	\end{equation}
	
	Let  $(\pi^{-n},\la^{-n},h^{-n})$ be a rectangle representation of $(\pi^{-n},\la^{-n},\tau^{-n})$ and $X^{-n}$ be a vertex of  $ (\pi^{-n},\la^{-n},\tau^{-n})$ 
	 %corresponding to $S\in\Sigma$ 
	 such that $\ell$ is a horizontal segment whose left endpoint is $X^{-n}$. Recall that in the  rectangle representation, all vertices of $(\pi^{-n},\la^{-n},h^{-n})$  lie on the left-hand side vertical sides of  rectangles. Let $\al\in\mathcal A$ be such that $X^{-n}\in \{\sum_{\pi_0(\beta)<\pi_0(\al)}|\la_{\beta}|\}\times [0,h^{-n}_\al)$. By \eqref{biglambda}, we have that 
%	\begin{equation}\label{laplace}
%	|\ell|<|\la^{-n}_{\al}|,
%	\end{equation}
%	which implies that 
	\[
	\ell\subset\left[\sum_{\pi_0(\beta)<\pi_0(\al)}|\la_{\beta}|,\sum_{\pi_0(\beta)\le\pi_0(\al)}|\la_{\beta}|\right)\times [0,h^{-n}_\al),
	\]
	that is $\ell$ is wholly included inside the rectangle corresponding to $\al$. In particular the strict inequality in \eqref{biglambda} implies that the right-hand side endpoint of $\ell$ is in the interior of this rectangle. Thus it cannot be a vertex of  $(\pi^{-n},\la^{-n},h^{-n})$ and, in particular, it cannot be an element of $\Sigma$ which is a contradiction.
				\begin{figure}
		\includegraphics[scale=0.5]{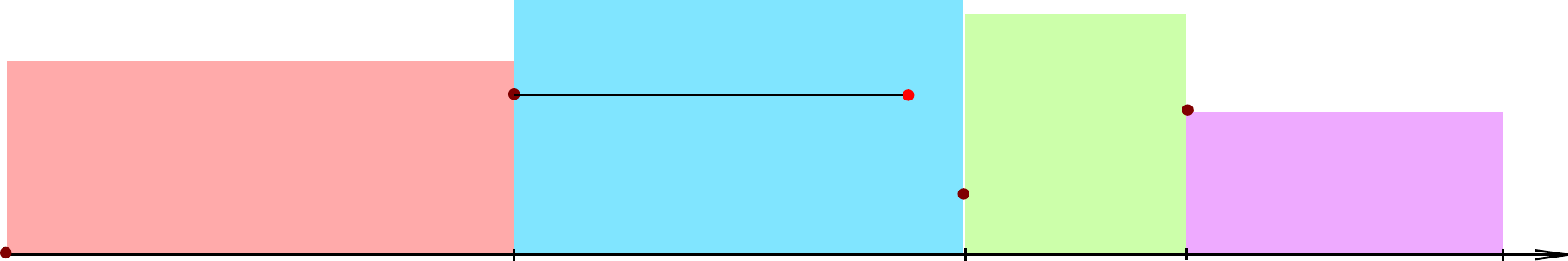}
		%		$\quad\quad\quad\quad$
		%		\includegraphics[scale=0.25]{drawing-1.pdf}
		\caption{The rectangle cannot grow wider if the surface has a horizontal saddle connection.}
	\end{figure}
	\end{proof}

\section{Horizontal connections and minimality}

% (or more precisely, the last and second from last vertex of upper or lower curve would be  respectively under or over the $x$-axis). 
In the Lemma \ref{win0} we described an easy condition for a symbol to stop being a backward winner. However, it is not the only possibility for a symbol to stop winning. Indeed, it appears that this phenomenon can be also observed in translation surfaces with horizontal cylinders.

\noindent\textbf{Example.} Consider $\#\mathcal A=4$ and $\pi\in S_0^{\mathcal A}$ given by.
\[
\pi_1\circ \pi_0^{-1}(i)=5-i\ \text{for}\ i=1,2,3,4.
\]
Consider moreover $\la\in\re_{>0}^\mathcal A$ and $\tau\in\Theta^{\mathcal A}$ such that $\tau_{\pi_0^{-1}(2)}=-\tau_{\pi_0^{-1}(3)}>0$ (see Figure \ref{ind}). 

We claim that
\begin{equation}\label{4win}
\text{neither $\pi_0^{-1}(2)$ nor ${\pi_0^{-1}(3)}$ can  be ever backward winners. }
\end{equation}
Indeed, if $\pi_0^{-1}(1)$ is the backward winner for  $(\pi,\la,\tau)$ then since $\tau_{\pi_0^{-1}(2)}>0$, it is also a backward winner for $(\pi^{-1},\la^{-1},\tau^{-1})$. Moreover, since  $\tau_{\pi_0^{-1}(2)}=-\tau_{\pi_0^{-1}(3)}$, we get $\sum_{\al\in\mathcal{A}}\tau^{-2}_\al=\sum_{\al\in\mathcal A}\tau_{\al}$ and thus $\pi_0^{-1}(1)$ is the backward winner for $(\pi^{-2},\la^{-2},\tau^{-2})$ and $\pi^{-3}=\pi$. We can get an analogous conclusion if $\pi_0^{-1}(4)$ is an initial winner. Note that in the 3  steps of backward Rauzy-Veech induction described above, $\pi_0^{-1}(2)$ and ${\pi_0^{-1}(3)}$ did not win backward. Hence, since $\pi^{-3}=\pi$, by induction we obtain \eqref{4win} (see figure \ref{ind}).
			\begin{figure}[h]
	\includegraphics[scale=0.6]{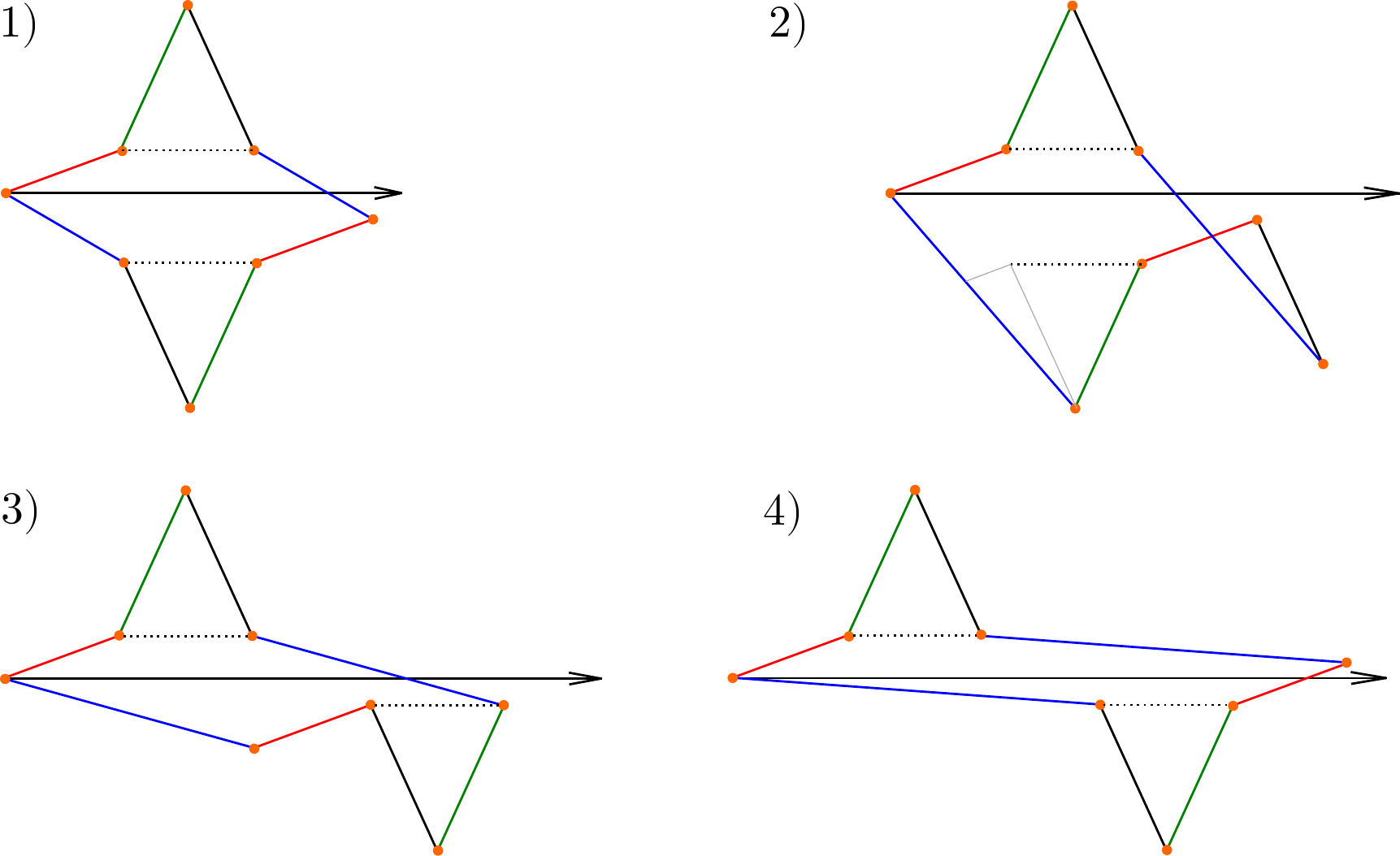}
	%		$\quad\quad\quad\quad$
	%		\includegraphics[scale=0.25]{drawing-1.pdf}
	\caption{Initial three steps of backward Rauzy-Veech induction of a translation surface with $\pi_0^{-1}(2)$ and ${\pi_0^{-1}(3)}$ never winning, while the values of $\tau_{\pi_0^{-1}(2)}$ and $\tau_{\pi_0^{-1}(3)}$ are non-zero}\label{ind}
\end{figure}

Note that the surface described above has a horizontal cylinder\footnote{A \emph{horizontal cylinder of length $\kappa$} is a maximal subsurface $C$ foliated by periodic orbits of length $\kappa$.}   starting at the side corresponding to $\pi_0^{-1}(2)$ of length $\la_{\pi_0^{-1}(2)}+\la_{\pi_0^{-1}(3)}$. In particular, the horizontal flow is not minimal. The following fact connects this observation with the parametrical occurrence of horizontal saddle connections. 
\begin{theorem}\label{min}
	Let $\mathcal A$ be an alphabet of $d\ge 3$ elements. If $(\pi^{-n},\la^{-n},\tau^{-n})$ is properly defined for all $n\in\n$, then the horizontal flow $(T_t)_{t\in\re}$ on $(\pi,\la,\tau)$ is minimal if and only if  there is $N\in\n$ such that the set of sides of the polygon $(\pi^{-N},\la^{-N},\tau^{-N})$ contains all (if any) horizontal connections.
\end{theorem}
%Before proving the above result, let us state first the following auxiliary lemma.
Before proving Theorem \ref{min}, let us state a result concerning the decay of the vertical parameters of polygonal representations of  a translation surface as we act by the backward Rauzy-Veech induction algorithm. In \cite{MUY} the authors proved the following result.
\begin{proposition}[Lemma A.8 in \cite{MUY}]\label{decre1}
	Assume that the surface $(\pi,\la.\tau)$ has no horizontal saddle connections. Then there exists an increasing sequence $\{n_k\}_{k\in\n}$ of positive integers such that 
		\[
	\lim_{k\to\infty}\max_{\al\in\mathcal A}\tau^{-n_k}_\al=0.
	\]
\end{proposition}
\begin{corollary}\label{decre}
	Assume that $(\pi,\la,\tau)$ is such that its only horizontal connections (if any) are sides of the polygonal representations. Then there exists an increasing sequence $\{n_k\}_{k\in\n}$ of positive integers such that 
	\[
	\lim_{k\to\infty}\max_{\al\in\mathcal A}\tau^{-n_k}_\al=0.
	\]
\end{corollary}
\begin{proof}
	If there are no saddle connections, then the statement of the corollary follows directly from Proposition \ref{decre1}. Let then $1\le k\le \#\mathcal A$ and $\beta_1,\ldots,\beta_k\in\mathcal A$ be such that $\tau_{\beta_i}=0$ for all $i=1,\ldots,k$. Consider a surface $(\pi,\tilde\la,\tilde\tau)$ obtained by setting $\tilde\tau=\tau$ and
	\[
	\tilde\la_\al=\begin{cases}
	0 & \text{ iff }\al=\beta_i\text{ for some }i=1,\ldots,k\\
	\la_\al \text{ otherwise }
	\end{cases}
	\]
	Since $\tau_{\beta_i}=0$ for all $i=1,\ldots,k$, the surface $(\pi,\la',\tau)$ is indeed properly defined\footnote{This surface is not an element of $S_0^{\mathcal A}\times\Lambda\mathcal A\times\Theta^{\mathcal A}$. However the backward Rauzy-Veech algorithm can be easily extended to surfaces with some of the horizontal parameters vanishing.} and via assumptions does not have horizontal saddle connections. Thus by Proposition \ref{decre1} for every $\ep>0$ there exists $n_\ep\in\n$ such that $(\pi^{n_\ep},\tilde\la^{n_\ep},\tilde\tau^{n_\ep})$ satisfies
	\[
	\max_{\al\in\mathcal A}|\tilde\tau^{-n_k}_\al|<\ep.
	\]
	It suffices now to see that to obtain the polygonal representation of $(\pi^{n_\ep},\la^{n_\ep},\tau^{n_\ep})$  from $(\pi^{n_\ep},\tilde\la^{n_\ep},\tilde\tau^{n_\ep})$  we only ``extend" the sides of  $(\pi^{n_\ep},\tilde\la^{n_\ep},\tilde\tau^{n_\ep})$   horizontally and not vertically (see Figure \ref{ext}). 
				\begin{figure}[h]
		\includegraphics[scale=1.1]{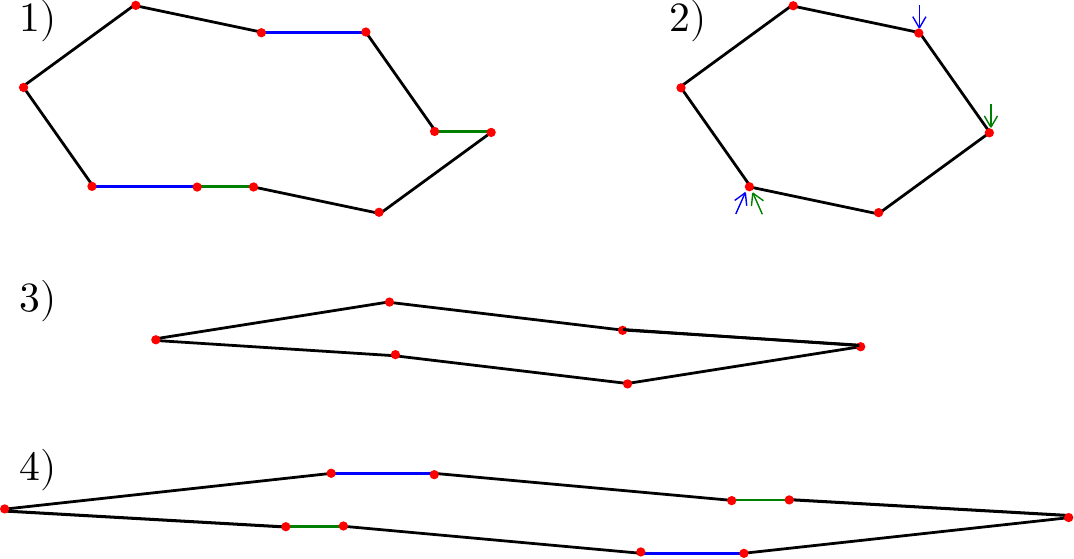}
		%		$\quad\quad\quad\quad$
		%		\includegraphics[scale=0.25]{drawing-1.pdf}
		\caption{To get to the case where there is no horizontal saddle connections we first contract them, then act by the backward Rauzy-Veech algorithm to finally extend the saddle connections back as well as the intervals corresponding to the symbols which won against saddle connections.}\label{ext}
	\end{figure}
	Indeed, by Lemma \ref{win0}, none of the symbols $\beta_1,\ldots,\beta_k$ wins, hence 
	\[
	\la ^{n_\ep}_{\beta_i}=\la_{\beta_i}\ \text{and}\ \tau^{n_\ep}_{\beta_i}=0\ \text{for every}\ i=1,\ldots,k.
	\]
	On the other hand if $\al\neq\beta_i$ for every $i=1,\ldots,k$, then by \eqref{matback} we have 
	\[
	\la ^{n_\ep}_{\al}=\tilde\la^{n_\ep}_\al+\sum_{i=1}^{k}A^{-n_\ep}_{\al\beta_i}(\pi,\la,\tau)\la_{\beta_i},
	\]
	but
		\[
	\tau ^{n_\ep}_{\al}=\tilde\tau^{n_\ep}_\al+\sum_{i=1}^{k}A^{-n_\ep}_{\al\beta_i}(\pi,\la,\tau)\tau_{\beta_i}=\tilde\tau^{n_\ep}_\al,
	\]
	thus 
		\[
	\max_{\al\in\mathcal A}|\tau^{-n_\ep}_\al|<\ep,
	\]
	which proves the corollary.
	\end{proof}
We can now prove Theorem \ref{min}.
\begin{proof}[Proof of Theorem \ref{min}.] Note first that in view of Lemma \ref{orb0} the rightwards horizontal separatrix starting at point $0:=(0,0)$, which  we denote by $\g_0:=(T_t(0))_{t>\ge 0}$, is infinite, i.e it is not a saddle connection. Note moreover that

	$(``\Rightarrow")$ Assume now that the horizontal flow is minimal. Then the rightwards separatrix is dense in $(\pi,\la,\tau)$. Let us show that this forces all horizontal saddle connections as sides of the polygonal representation after sufficient number of backward Rauzy-Veech induction steps. Suppose by contradiction that one of the horizontal saddle connections does not appear as a vertical segment for infinitely many $n>0$ in $(\pi^{-n},\la^{-n},\tau^{-n})$. 
	Then by Lemma \ref{win0}, it does not appear as a vertical segment for any $n>0$. By Theorem \ref{noinfty} there exists $\al$ such that $\tau_\al^{-n}\neq 0$ for all $n\in\n$ and $\al$ is never a backward winner (note that $\al$ can win finitely many times, then we renumerate the steps of backward Rauzy-Veech induction). Since by Remark \ref{hitby0} the separatrix $\g_0$ does not hit the interval corresponding to $\al$, we obtain that $\g_0$ cannot pass through the interior of  the triangle given by vertices
	\[
	\begin{split}
	&a(\pi^{-n},\la^{-n},\tau^{-n},\pi_0(\al)-1);\\
	(\sum_{i\le\pi_0^{-n}(\al)}&\la_{(\pi_0^{-n})^{-1}(i)},\sum_{i\le\pi_0^{-n}(\al)}\tau_{(\pi_0^{-n})^{-1}(i)});\\
	&a(\pi^{-n},\la^{-n},\tau^{-n},\pi_0(\al)),
	\end{split}
	\]
	if $\tau_\al>0$ or 
		\[
		\begin{split}
	&a(\pi^{-n},\la^{-n},\tau^{-n},\pi_0(\al)-1);\\
(\sum_{i\le\pi_0^{-n}(\al)}&\la_{(\pi_0^{-n})^{-1}(i)},\sum_{i\le\pi_0^{-n}(\al)}\tau_{(\pi_0^{-n})^{-1}(i)});\\
&a(\pi^{-n},\la^{-n},\tau^{-n},\pi_0(\al)),
	\end{split}
	\]
	if $\tau_\al<0$, where the definition of the triangle does not depend on $n$ (one can also swap in the above definitions $\pi_0$ to $\pi_1$ and $a$ to $b$),  see Figure \ref{triangle}. Since this triangle is of positive Lebesgue measure, this contradicts the minimality of the horizontal flow.
	
				\begin{figure}[h]
		\includegraphics[scale=1.1]{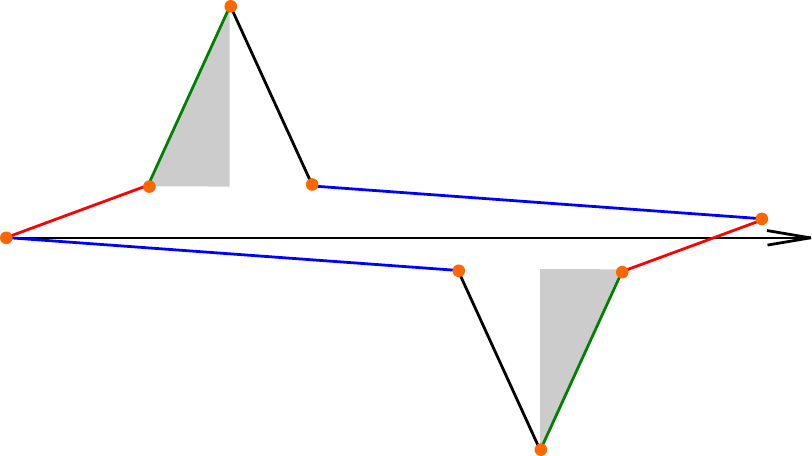}
		%		$\quad\quad\quad\quad$
		%		\includegraphics[scale=0.25]{drawing-1.pdf}
		\caption{The horizontal separatrix starting from the origin does not visit the shaded triangles.}\label{triangle}
	\end{figure}
	
	$(``\Leftarrow")$ Assume now that there exists $N\in \n$ such that for all $n>N$ all horizontal saddle connections are sides of the polygon $(\pi^{-n},\la^{-n},\tau^{-n})$. We want to show that this implies minimality. By Proposition \ref{decre} there exists a sequence $\{n_k\}_{k\in\n}$ such that 
	\[
	\lim_{k\to\infty}\max_{\al\in\mathcal A}{\tau^{-n_k}_\al}=0. 
	\]
	Fix $\ep>0$ and let $k'\in\n$ be big enough so that $\max_{\al\in\mathcal A}{\tau^{-n_k'}_\al}<\ep/d$. Then $(\pi^{-n_k'},\la^{-n_k'},\tau^{-n_k'})$ seen as a polygon in $\re^2$ is included as a subset in a rectangle $[-\ep,\ep]\times [0,\sum_{\al\in\mathcal A}\la^{-n_k'}_\al]$. Hence every point in $(\pi^{-n_k'},\la^{-n_k'},\tau^{-n_k'})=(\pi,\la,\tau)$ is at most $\ep$-far from $\g_0$ which yields the density of $\g_0$. However by Corollary \ref{minorb} this is equivalent to the minimality of the horizontal flow, which finishes the proof.
	
				\begin{figure}[h]
		\includegraphics[scale=0.8]{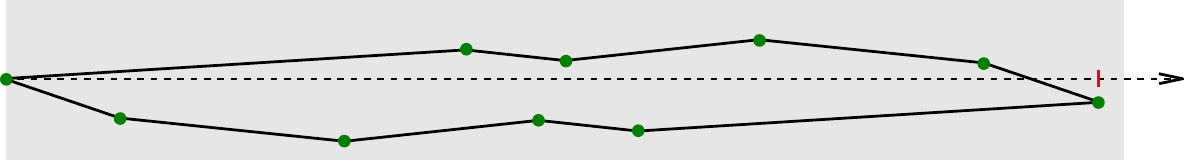}
		%		$\quad\quad\quad\quad$
		%		\includegraphics[scale=0.25]{drawing-1.pdf}
		\caption{The rectangle with a separatrix segment as its axis contains the whole surface.}
	\end{figure}

	\end{proof}

\end{document}